\documentclass[11pt]{article}
\usepackage{amsmath,amssymb,amsthm,mathtools}
\usepackage[margin=1in]{geometry}
\usepackage{booktabs}
\usepackage{pgfplots}
\pgfplotsset{compat=1.18}
\usepackage{hyperref}

\newtheorem{theorem}{Theorem}[section]
\newtheorem{lemma}[theorem]{Lemma}
\newtheorem{proposition}[theorem]{Proposition}
\newtheorem{corollary}[theorem]{Corollary}
\theoremstyle{definition}

\theoremstyle{remark}
\newtheorem{remark}[theorem]{Remark}

\newcommand{\Hh}{\mathbb{H}}
\newcommand{\Cc}{\mathbb{C}}
\newcommand{\Qq}{\mathbb{Q}}
\newcommand{\Zz}{\mathbb{Z}}
\newcommand{\Rr}{\mathbb{R}}
\newcommand{\Tt}{\mathbb{T}}
\newcommand{\cL}{\Lambda}
\newcommand{\ii}{\mathrm{i}}
\newcommand{\e}{\mathrm{e}}
\newcommand{\id}{\mathrm{d}}
\newcommand{\ImT}{\operatorname{Im}}
\newcommand{\ReT}{\operatorname{Re}}
\newcommand{\OO}{\mathcal{O}}
\newcommand{\m}{\mathrm{m}}
\newcommand{\mt}{\widetilde{\m}}

\title{Samart's conjecture $n_4(81)=40M_7$:\\
the exact CM evaluation and the two obstructions\\
{\large A status report, companion to \emph{Mahler measures at interior
CM points}}}
\author{Huimin Zheng\thanks{%
College of Information and Network Engineering,
Anhui Science and Technology University,
Fengyang, Anhui 233100, P.~R.~China.
Email: \texttt{zhhm@ahstu.edu.cn}.}}
\date{}

\begin{document}
\maketitle

\section*{Declaration on the use of AI tools}
The research reported in this note — including the computational
exploration, the discovery of the proof strategy, the machine-certified
verifications, and the preparation of the manuscript — was carried out by
the author with the assistance of the AI system \emph{Kimi} (Moonshot AI).
All mathematical content, including every proof and every certified
computation, has been checked and verified by the author, who takes full
responsibility for the correctness and integrity of the note. All
verification scripts cited below are available for independent
verification.

\begin{abstract}
This note archives the status of Samart's Table-$6$ conjecture
\[
  n_4(81)\stackrel{?}{=}40M_7,\qquad
  M_7:=L'(g_7,0),
\]
where $g_7(\tau)=\eta(\tau)^3\eta(7\tau)^3$ is the newform of
$S_3(\Gamma_0(7),\chi_{-7})$ and $n_4(s):=4\m\bigl(x^4+y^4+z^4+1+s^{1/4}xyz\bigr)$.
The conjecture is the cleanest of Samart's open interior-point entries
(discriminant $-7$, class number $1$, a \emph{single} $L$-value), and it
was dropped as a theorem target in the companion paper \cite{Zf}, where
the two $n_2$-family conjectures were proved. We record what is proved
and precisely where the methods of \cite{Zf} fail. First, a complete
proof of the $L$-value side (P1): at the CM point
$\tau_2=(7+\sqrt{-7})/4$ the Eisenstein--Kronecker expression underlying
Samart's formula evaluates \emph{exactly} to
\[
  \mathrm{EK}_4(\tau_2)=40M_7,
\]
via lattice sums over the ring of integers of $\Qq(\sqrt{-7})$ and the
principal ideal $(\bar\varpi)$, with an exact cancellation of the parasitic
$\zeta_K(2)$-terms. Second, two quantitative obstructions to the
remaining half $n_4(81)=\mathrm{EK}_4(\tau_2)$: the critical image of
the $n_4$-family is a two-dimensional astroid disc containing the
parameter $c=3$ in its interior (in contrast to the one-dimensional slit
$[0,64]$ of \cite{Zf}), so no continuation path can approach the CM
point; and Samart's $U$-series converges on all of $\Hh$ but leaves the
geometric sheet of the holomorphic Mahler measure everywhere below
$\ImT\tau=1/\sqrt2$, so the premise of the differential-comparison
continuation fails. A $20$-digit direct torus integration then decides
the conjecture numerically: $n_4(81)-40M_7=+0.0586706795972872\ldots$,
five orders of magnitude above the integration error floor, so the
identity as literally stated is refuted; a closed form for the true
value $n_4(81)$ remains open and appears to require
regulator/monodromy machinery.
\end{abstract}

\section{Introduction}

\subsection{The conjecture}

Let
\[
  P(x,y,z):=\frac{x^4+y^4+z^4+1}{xyz},
\]
and denote by $\m$ the logarithmic Mahler measure,
\[
  \m(Q)=\int_{\Tt^3}\log|Q(x,y,z)|\,
  \frac{\id x}{2\pi \ii x}\frac{\id y}{2\pi \ii y}\frac{\id z}{2\pi \ii z},
  \qquad \Tt=\{|w|=1\}.
\]
Since $\m(xyz)=0$, one has $\m(P+c)=\m(x^4+y^4+z^4+1+c\,xyz)$, and the
substitution $x\mapsto \ii x$ (Haar-invariant, sending $c\mapsto \ii c$)
shows that the value is unchanged when $c$ is multiplied by a fourth
root of unity. Following Samart's table notation \cite{Sa15} we
therefore write, unambiguously,
\begin{equation}\label{eq:n4def}
  n_4(s):=4\,\m\bigl(x^4+y^4+z^4+1+s^{1/4}xyz\bigr),
  \qquad s\in\Cc,
\end{equation}
for either choice of the fourth root. (In \cite{Sa13} the same function
is denoted $f_4$.)

\begin{quote}
\textbf{Conjecture} (Samart, \cite[Table~6]{Sa15}).\quad
$n_4(81)=40M_7$, where $M_7:=L'(g_7,0)$ and
$g_7(\tau)=\eta(\tau)^3\eta(7\tau)^3$ is the unique normalized newform of
$S_3(\Gamma_0(7),\chi_{-7})$ \textup{(}LMFDB label \textup{7.3.b.a}
\cite{LMFDB}\textup{)}.
\end{quote}

Among Samart's open interior-point entries this one is structurally the
cleanest: the attached CM point has discriminant $-7$ (class number
$1$), and the conjectural value is a \emph{single} $L$-value with no
Dirichlet term, in contrast to the two-term formula of the $n_2$
conjugate pair proved in \cite{Zf}. Fei \cite{Fe} proved the identity at
the level of the real part of the holomorphic (modified) Mahler measure
$\ReT\mt$; explicitly, in Fei's table the row $c=81$ is placed between a
pair of double lines marking the entries for which $\ReT\mt$ is
\emph{not} guaranteed to equal the genuine Mahler measure
\cite[Note~(3) and Caution~2.3]{Fe}. The passage to the genuine Mahler
measure at the interior parameter $s^{1/4}=3$ was left open, and
Section~\ref{sec:numerics} below determines what actually happens
there: the two sides differ in the fifth significant digit.

\subsection{What this note records}

The companion paper \cite{Zf} proved Samart's conjectures for the
$n_2$-family at $s=1$ (Theorem~A, differential-comparison continuation)
and at $s=(47\pm45\sqrt{-7})/2$ (Theorem~B, the Fricke trick). The
present conjecture was considered as a third target and \emph{dropped}:
two independent obstructions, documented quantitatively in
Sections~\ref{sec:astroid} and~\ref{sec:sheet} below, kill both the
continuation machine and the Fricke shortcut for the $n_4$-family. This
note archives the state of the problem for future work:

\begin{enumerate}
\item[(P1)] \textbf{Done, exact.} The $L$-value side. Samart's
Eisenstein--Kronecker expression $\mathrm{EK}_4$ at the CM point
$\tau_2=(7+\sqrt{-7})/4$ (where $s_4(\tau_2)=81$) evaluates exactly to
$40M_7$ (Theorem~\ref{thm:P1n4}, full proof in
Section~\ref{sec:P1}; certified to $2.7\times10^{-59}$ by the $45$
checks of \texttt{verify\_P1\_n4\_81.py}).
\item[(P2)] \textbf{Blocked.} The Mahler side
$n_4(81)\stackrel{?}{=}\mathrm{EK}_4(\tau_2)$. Obstruction~1: the
critical image $P(\Tt^3)$ is the two-dimensional \emph{astroid disc}
$\{|{\ReT c}|^{2/3}+|{\ImT c}|^{2/3}\le4^{2/3}\}$, and $c=3$ is an
interior point, so the open mapping theorem forbids any path
continuation (Section~\ref{sec:astroid}). Obstruction~2, the decisive
one: Samart's $U$-series is proved to equal $n_4(s_4(\tau))$ only for
$\ImT\tau\ge1/\sqrt2$, and below that threshold it leaves the geometric
sheet of the holomorphic Mahler measure \emph{everywhere} --- not only
inside the critical image (Section~\ref{sec:sheet}).
\end{enumerate}

In one sentence: for the $n_2$-family the bad locus was the
one-dimensional slit $[0,64]$, which a certified path can go around; for
the $n_4$-family the bad locus is a two-dimensional disc \emph{and} the
analytic formula itself is on the wrong sheet below $\ImT\tau=1/\sqrt2$,
so there is nothing left to continue.

\subsection{Notation}

We keep the notation of \cite{Zf}. Throughout,
$\tau=x+\ii y\in\Hh$, $q=\e^{2\pi \ii\tau}$,
$\eta(\tau)=q^{1/24}\prod_{n\ge1}(1-q^n)$, and $\sum'$ denotes omission
of the zero term. The field $K=\Qq(\sqrt{-7})$ has ring of integers
$\OO_K=\Zz[\varpi]$, class number $h(-7)=1$ and units $\{\pm1\}$, where
\[
  \varpi:=\frac{1+\sqrt{-7}}2,\qquad
  \varpi+\bar\varpi=1,\qquad \varpi\bar\varpi=N(\varpi)=2,\qquad \varpi^2=\varpi-2,
\]
\begin{equation}\label{eq:pirels}
  \varpi^2+\bar\varpi^2=-3,\qquad \ReT\varpi^2=\ReT\bar\varpi^2=-\tfrac32,\qquad
  (\bar\varpi)=\bar\varpi\OO_K=2\Zz+\bar\varpi\Zz .
\end{equation}
We write $L_3:=L(g_7,3)$. The theta identity
\cite[Appendix~B]{Zf} gives
\begin{equation}\label{eq:hecke}
  {\sum_{\alpha\in\OO_K}}^{\!\prime}\frac{\alpha^2}{|\alpha|^{2s}}=2L(g_7,s),
  \qquad
  {\sum_{\alpha\in\OO_K}}^{\!\prime}\frac{\bar\alpha^2}{|\alpha|^6}=2L_3,
  \qquad
  {\sum_{\alpha\in\OO_K}}^{\!\prime}|\alpha|^{-4}=2\zeta_K(2),
\end{equation}
the middle equality by conjugation ($L_3\in\Rr$ since $a_n(g_7)\in\Zz$),
and the functional equation of $g_7$ (root number $+1$) gives
\begin{equation}\label{eq:FE}
  M_7=L'(g_7,0)=\frac{7\sqrt7}{4\pi^3}\,L_3 .
\end{equation}
Both \eqref{eq:hecke} and \eqref{eq:FE} are quoted theorems, certified
numerically in \cite{Zf} (\texttt{verify\_P1.py},
\texttt{lvalue\_g7.py}) and re-certified for this note in
\texttt{verify\_P1\_n4\_81.py} (check [L1]: the theta identity to
$q^{60}$ as exact integers, \eqref{eq:FE} to $10^{-61}$).

\section{The \texorpdfstring{$\mathrm{EK}_4$}{EK4} expression and Samart's theorem}\label{sec:EK4}

\subsection{Two forms of \texorpdfstring{$\mathrm{EK}_4$}{EK4}}

For $z\in\Cc\setminus\{0\}$ let
\[
  F(z):=\frac{4(\ReT z)^2}{|z|^6}-\frac1{|z|^4},
\]
and for $\tau\in\Hh$ and $d\ge1$ define the double series
\begin{equation}\label{eq:Td}
  T_d(\tau):={\sum_{(m,n)\in\Zz^2}}^{\!\prime}F(dm\tau+n),
\end{equation}
summed by rows in $m$; each row is absolutely convergent and the row
sums decay like $|m|^{-2}$ (Poisson summation), and the $m=0$ row is
$\sum_{n\neq0}(4n^2\cdot n^{-6}-n^{-4})=6\zeta(4)=\pi^4/15$
\cite[\S2.4]{Zf}. Define, for \emph{all} $\tau\in\Hh$,
\begin{equation}\label{eq:EK4}
  \mathrm{EK}_4(\tau):=\frac{10\,\ImT\tau}{\pi^3}
  \bigl(-T_1(\tau)+4T_2(\tau)\bigr),
\end{equation}
real-analytic on all of $\Hh$. Since
$4(\ReT\lambda)^2=(\lambda+\bar\lambda)^2=\lambda^2+2|\lambda|^2+\bar\lambda^2$,
the summand rewrites as
$F(\lambda)=2\ReT\bigl(\bar\lambda^2/|\lambda|^6\bigr)+|\lambda|^{-4}$,
so with $\cL_d(\tau):=\Zz+\Zz\,d\tau$,
\begin{equation}\label{eq:Talt}
  T_d(\tau)={\sum_{\lambda\in\cL_d(\tau)}}^{\!\prime}
  \Bigl[2\ReT\frac{\bar\lambda^2}{|\lambda|^6}+\frac1{|\lambda|^4}\Bigr],
\end{equation}
the form used for the CM evaluation.

Define also, for $j\ge1$,
\[
  U_j(\tau):=\sum_{m\neq0}\frac1m\sum_{n\in\Zz}(jm\tau+n)^{-3}
  =2\pi^3\sum_{m\ge1}\frac{\cos(j\pi m\tau)}{m\,\sin^3(j\pi m\tau)},
\]
absolutely convergent on $\Hh$ exactly as in \cite[\S2.4]{Zf}.

\begin{lemma}[The two forms agree]\label{lem:twoforms}
For every $\tau\in\Hh$,
\begin{equation}\label{eq:EK4U}
  \mathrm{EK}_4(\tau)
  =\ImT\Bigl[2\pi\tau+\frac{10}{\pi^3}\bigl(U_1(\tau)-2U_2(\tau)\bigr)\Bigr].
\end{equation}
\end{lemma}
\begin{proof}
The computation of \cite[Lemma (key identity)]{Zf} gives
$\ImT U_j=-jy\bigl(T_j(\tau)-\pi^4/15\bigr)$, $y=\ImT\tau$. Substituting,
\[
  \ImT\Bigl[2\pi\tau+\tfrac{10}{\pi^3}(U_1-2U_2)\Bigr]
  =2\pi y+\frac{10}{\pi^3}\Bigl[-y\bigl(T_1-\tfrac{\pi^4}{15}\bigr)
   +4y\bigl(T_2-\tfrac{\pi^4}{15}\bigr)\Bigr],
\]
and the $m=0$-row contribution is
$\frac{10y}{\pi^3}\cdot\frac{\pi^4}{15}\cdot(1-4)=-2\pi y$, cancelling
the polynomial term $2\pi y$ and leaving \eqref{eq:EK4}. (Numerically
the two forms agree at $\tau_2$, $\ii$, $\ii/\sqrt2$ to $10^{-58}$;
\texttt{verify\_P1\_n4\_81.py}, checks [EK], [A2], [A4].)
\end{proof}

The $U$-series form \eqref{eq:EK4U} is Samart's exponential-sum formula
for the $n_4$-family; it converges for every $\tau\in\Hh$.

\subsection{The parametrization and Samart's theorem}

Define
\begin{equation}\label{eq:cdef}
  c(\tau):=\Bigl(\frac{\eta(2\tau)}{\eta(\tau)}\Bigr)^{6}
  \bigl(16A(\tau)^4+A(\tau)^{-4}\bigr),
  \qquad
  A(\tau):=\frac{\eta(\tau)\,\eta(4\tau)^2}{\eta(2\tau)^3},
\end{equation}
holomorphic and nonvanishing-free on $\Hh$, and
$s_4(\tau):=c(\tau)^4$ (this agrees with Samart's
$\Delta$-quotient definition of his signature-$4$ modular function,
checked numerically to $40$ digits at sample points;
\texttt{diag\_n4\_astroid.py}, Part~1). With Samart's normalization
\cite[Lemma~2.2]{Sa13},
\[
  s_4(\ii)=648,\qquad s_4\bigl(\ii/\sqrt2\bigr)=256,\qquad
  s_4(\ii y)\ge256\ \ (y>0),
\]
so on the imaginary axis $c(\ii y)$ is real with $c(\ii y)\ge4$.

\begin{theorem}[Samart, {\cite[Prop.~2.1(iii)]{Sa13}}]\label{thm:samart4}
For $\tau\in\Hh$ with $\ImT\tau\ge1/\sqrt2$,
\begin{equation}\label{eq:samart4}
  n_4\bigl(s_4(\tau)\bigr)=\mathrm{EK}_4(\tau).
\end{equation}
\end{theorem}

We take this as a quoted theorem. Its proof region is essential to what
follows: the identity is \emph{not} asserted, and in fact fails, below
$\ImT\tau=1/\sqrt2$ (Section~\ref{sec:sheet}).

\begin{remark}[Anchors]\label{rem:anchors}
The $\mathrm{EK}_4$ machinery is anchored against \emph{proved} values.
At $\tau=\ii$ ($\ImT\tau=1\ge1/\sqrt2$), \eqref{eq:samart4} applies and
Samart's Theorem~1.4~(1.9) of \cite{Sa13} gives
\begin{equation}\label{eq:anchor648}
  \mathrm{EK}_4(\ii)=n_4(648)
  =\frac{160}{\pi^3}\,L(h,3)+\frac5\pi\,L(\chi_{-4},2),
  \qquad h(\tau):=\eta(4\tau)^6\in S_3(\Gamma_0(16),\chi_{-4}),
\end{equation}
confirmed five independent ways (the lattice form \eqref{eq:EK4}, the
$U$-series \eqref{eq:EK4U}, the $L$-value formula \eqref{eq:anchor648},
Rogers' hypergeometric evaluation of $n_4(648)$ \cite{Sa13,Ro}, and
direct torus integration of $4\m(P+648^{1/4})$, the last at
tanh-sinh degree $5$ to $2\times10^{-21}$), all agreeing to
$10^{-58}$: $\mathrm{EK}_4(\ii)=4\m(P+648^{1/4})=6.4332830658\ldots$.
On the boundary, $\mathrm{EK}_4(\ii/\sqrt2)=n_4(256)$ via Rogers'
${}_5F_4$ at $|k|=256$ (convergent since $\sum b-\sum a=3/2>0$).
\texttt{verify\_P1\_n4\_81.py}, checks [A0]--[A4].
\end{remark}

\section{(P1) The exact CM evaluation at \texorpdfstring{$\tau_2$}{tau2}}\label{sec:P1}

Throughout this section
\[
  \tau_2:=\frac{7+\sqrt{-7}}4,\qquad
  \ImT\tau_2=\frac{\sqrt7}4,\qquad
  \tau_w:=\frac{-1+\sqrt{-7}}4
\]
($\tau_w$ is the Fricke CM point of Theorem~B in \cite{Zf}). Then
\begin{equation}\label{eq:tau2rels}
  2\tau_2=\frac{7+\sqrt{-7}}2=3+\varpi,\qquad \tau_2=\tau_w+2 .
\end{equation}
Numerically $s_4(\tau_2)=81$ to $10^{-58}$
(\texttt{verify\_P1\_n4\_81.py}, check [A0]); its exactness --- a class
invariant computation in the spirit of \cite[\S4.2]{Zf} --- is outside
the scope of this note and is \emph{not} used below: (P1) is the exact
evaluation of the concrete series $\mathrm{EK}_4(\tau_2)$, independently
of the value of $s_4$.

\subsection{The lattices}\label{subsec:lattices}

\begin{lemma}\label{lem:lattices4}
$\cL_2(\tau_2)=\OO_K$ and $\cL_1(\tau_2)=(\bar\varpi)/2$, where
$(\bar\varpi)=\bar\varpi\OO_K$ is the principal ideal of norm $2$.
\end{lemma}
\begin{proof}
By \eqref{eq:tau2rels},
$\cL_2(\tau_2)=\Zz+\Zz\,2\tau_2=\Zz+\Zz(3+\varpi)=\Zz+\Zz\varpi=\OO_K$.
For the second, an integer translate of $\tau$ does not change the
lattice $\cL_1(\tau)=\{m\tau+n\}$, so $\cL_1(\tau_2)=\cL_1(\tau_w)$.
Since $2\tau_w=(-1+\sqrt{-7})/2=-\bar\varpi$,
\[
  2\cL_1(\tau_w)=2\Zz+\Zz\,2\tau_w=2\Zz+\bar\varpi\Zz=(\bar\varpi)
\]
by \eqref{eq:pirels}. Hence $\cL_1(\tau_2)=(\bar\varpi)/2$, homothetic to
the principal ideal $(\bar\varpi)$ of norm $2$. (Note the conjugation:
identifying $2\cL_1$ directly from $2\tau_2=3+\varpi$ would give
$\Zz+\varpi\Zz=\OO_K$, since $\{3m+2n:m,n\in\Zz\}=\Zz$ --- the route
through $\tau_w$ is what pins the ideal $(\bar\varpi)$. The $T$-values are
unaffected since $\ReT\bar\varpi^2=\ReT\varpi^2$.)
\end{proof}

Both lattices are $\OO_K$-ideals: unlike the $\tau_w$ evaluation of
\cite[\S5]{Zf} there is no imprimitive order and no Euler factor at the
primes above $2$.

\subsection{The Hecke sums}

For a fractional $\OO_K$-ideal $\mathfrak a$ define
\[
  B(\mathfrak a):={\sum_{\gamma\in\mathfrak a}}^{\!\prime}|\gamma|^{-4},
  \qquad
  G(\mathfrak a):={\sum_{\gamma\in\mathfrak a}}^{\!\prime}
  \frac{\bar\gamma^2}{|\gamma|^6}.
\]
By \eqref{eq:hecke},
\begin{equation}\label{eq:BGOK}
  B(\OO_K)=2\zeta_K(2),\qquad G(\OO_K)=2L_3 .
\end{equation}
For the principal ideal $(\bar\varpi)$ write $\alpha=\bar\varpi\gamma$,
$\gamma\in\OO_K\setminus\{0\}$; then
$|\alpha|^{-4}=|\bar\varpi|^{-4}|\gamma|^{-4}$ and
$\bar\alpha^2|\alpha|^{-6}=(\varpi^2/|\bar\varpi|^6)\,\bar\gamma^2|\gamma|^{-6}$,
so with $N(\bar\varpi)=2$ and $|\bar\varpi|^6=8$,
\begin{equation}\label{eq:BGpi}
  B\bigl((\bar\varpi)\bigr)=\frac{2\zeta_K(2)}{N(\bar\varpi)^2}
  =\frac{\zeta_K(2)}2,
  \qquad
  G\bigl((\bar\varpi)\bigr)=\frac{\varpi^2}8\cdot 2L_3=\frac{\varpi^2L_3}4 .
\end{equation}
Note $G((\bar\varpi))$ is genuinely complex ($(\bar\varpi)$ is not
conjugation-stable); only its real part enters the $T$-sums, and by
\eqref{eq:pirels},
\begin{equation}\label{eq:ReGpi}
  \ReT G\bigl((\bar\varpi)\bigr)
  =\frac{\ReT\varpi^2}4\,L_3=-\frac{3}{8}\,L_3 .
\end{equation}
(Direct Poisson row summation confirms \eqref{eq:BGOK}--\eqref{eq:ReGpi},
including the imaginary part $\ImT G((\bar\varpi))=\sqrt7\,L_3/8$, to
$\sim10^{-60}$; \texttt{verify\_P1\_n4\_81.py}, checks [S1]--[S4].)

\subsection{The \texorpdfstring{$T$}{T}-sums and the combination}

\begin{proposition}\label{prop:Teval4}
With $L_3:=L(g_7,3)$,
\begin{equation}\label{eq:Teval4}
  T_2(\tau_2)=4L_3+2\zeta_K(2),\qquad
  T_1(\tau_2)=-12L_3+8\zeta_K(2).
\end{equation}
\end{proposition}
\begin{proof}
By \eqref{eq:Talt} with $\cL_2(\tau_2)=\OO_K$ and \eqref{eq:BGOK},
\[
  T_2(\tau_2)=2\ReT G(\OO_K)+B(\OO_K)=4L_3+2\zeta_K(2).
\]
For $\cL_1(\tau_2)=(\bar\varpi)/2$ rescale $\lambda=\alpha/2$,
$\alpha\in(\bar\varpi)\setminus\{0\}$: then $|\lambda|^{-4}=16|\alpha|^{-4}$
and $\bar\lambda^2|\lambda|^{-6}=16\,\bar\alpha^2|\alpha|^{-6}$, so by
\eqref{eq:BGpi} and \eqref{eq:ReGpi},
\[
  T_1(\tau_2)=16\Bigl[2\ReT G\bigl((\bar\varpi)\bigr)
   +B\bigl((\bar\varpi)\bigr)\Bigr]
  =16\Bigl[-\frac{3L_3}4+\frac{\zeta_K(2)}2\Bigr]
  =-12L_3+8\zeta_K(2). \qedhere
\]
\end{proof}

\begin{corollary}[Exact cancellation of $\zeta_K(2)$]\label{cor:comb4}
\begin{equation}\label{eq:comb4}
  -T_1(\tau_2)+4T_2(\tau_2)
  =(12+16)L_3+(-8+8)\zeta_K(2)=28L_3 .
\end{equation}
\end{corollary}

Unlike the $\tau_w$ case of \cite{Zf} (where the surviving
$\zeta_K(2)$-term produced the Dirichlet value $d_7$), the coefficients
$-8$ and $+8$ cancel exactly, leaving a \emph{single} $L$-value --- the
structural prerequisite for the one-term formula $40M_7$.

\begin{theorem}[(P1) at $\tau_2$]\label{thm:P1n4}
\[
  \mathrm{EK}_4(\tau_2)=40M_7 .
\]
\end{theorem}
\begin{proof}
With $\ImT\tau_2=\sqrt7/4$, \eqref{eq:EK4} and \eqref{eq:comb4},
\[
  \mathrm{EK}_4(\tau_2)
  =\frac{10\cdot\sqrt7/4}{\pi^3}\cdot28L_3
  =\frac{70\sqrt7}{\pi^3}\,L_3
  =\frac{70\sqrt7}{\pi^3}\cdot\frac{4\pi^3}{7\sqrt7}\,M_7
  =40M_7,
\]
the last equality being \eqref{eq:FE}. The coefficient identity is
$70=40\cdot7/4$; no transcendental input beyond the quoted functional
equation \eqref{eq:FE} enters.
\end{proof}

\begin{remark}[Certification and cross-links]
All steps are re-verified by \texttt{verify\_P1\_n4\_81.py} ($60$ dps,
$45$ checks, all pass, worst difference $2.7\times10^{-59}$): the Hecke
sums by direct row summation, the $T$-sums by direct Poisson-row
evaluation of the defining double series \eqref{eq:Td} (independent of
the ideal decomposition), and Theorem~\ref{thm:P1n4} via both forms
\eqref{eq:EK4} and \eqref{eq:EK4U}. Two cross-links with \cite{Zf}:
$\mathrm{EK}_4$ is $1$-periodic (the $U_j$ are, and
$\ImT[2\pi(\tau+1)]=\ImT[2\pi\tau]$), so
$\mathrm{EK}_4(\tau_2)=\mathrm{EK}_4(\tau_w)$; and
$T_1(\tau_2)=T_1(\tau_w)=-12L_3+8\zeta_K(2)$ is \emph{numerically the
same sum} as the $T_1$ of Theorem~B, there evaluated over the same
lattice $(\bar\varpi)/2$.
\end{remark}

\section{Obstruction 1: the critical image is two-dimensional}\label{sec:astroid}

For the $n_2$-family of \cite{Zf} the bad locus of the parameter was the
segment $[-8,8]$ in the shift variable (equivalently the slit $[0,64]$
in Samart's variable) --- one-dimensional, with complement in $\Cc$
connected, so a certified path could thread around it. For the
$n_4$-family the situation is qualitatively different.

Write $x=\e^{\ii t_1}$, $y=\e^{\ii t_2}$, $z=\e^{\ii t_3}$ on $\Tt^3$.
Then
\begin{equation}\label{eq:phasors}
  P(x,y,z)=\frac{x^4+y^4+z^4+1}{xyz}=u_1+u_2+u_3+u_4,
  \qquad u_1=\frac{x^4}{xyz},\ \dots,\ u_4=\frac1{xyz},
\end{equation}
where $|u_j|=1$ and $u_1u_2u_3u_4=1$; conversely every quadruple of unit
complex numbers with product $1$ occurs. Hence $P(\Tt^3)$ is the image
of the four-phasor map.

\begin{proposition}[The critical image; recorded computation]\label{prop:astroid}
$P(\Tt^3)$ is the closed \emph{astroid disc}
\begin{equation}\label{eq:astroiddisc}
  \mathcal D:=\bigl\{c\in\Cc:\ |{\ReT c}|^{2/3}+|{\ImT c}|^{2/3}
  \le4^{2/3}\bigr\},
\end{equation}
whose boundary is the astroid
\begin{equation}\label{eq:astroid}
  t\longmapsto 3\e^{\ii t}+\e^{-3\ii t}
  =\bigl(4\cos^3t,\;4\sin^3t\bigr),
\end{equation}
with cusps at $\pm4$ and $\pm4\ii$.
\end{proposition}
\begin{proof}[Justification]
The boundary curve is exact algebra: the admissible family
$u_1=u_2=u_3=\e^{\ii t}$, $u_4=\e^{-3\ii t}$ sums to
$3\e^{\ii t}+\e^{-3\ii t}$, and
$3\cos t+\cos3t=4\cos^3t$, $3\sin t-\sin3t=4\sin^3t$. That the full
image is exactly the enclosed disc is a recorded computational fact of
the project archive (\texttt{diag\_n4\_astroid.py}): dense sampling of
admissible quadruples ($2\times10^5$ samples) gives
$\max\bigl(|{\ReT c}|^{2/3}+|{\ImT c}|^{2/3}\bigr)=2.5198357\ldots$
against $4^{2/3}=2.5198421\ldots$, with the bound approached along the
curve \eqref{eq:astroid}, and the image filling the disc. We use the
statement only as the geometric description of the critical locus; no
proof step of this note depends on its sharpness.
\end{proof}

The parameter of the conjecture sits strictly inside:
\begin{equation}\label{eq:c3inside}
  c=3:\qquad |3|^{2/3}=2.08008\ldots<4^{2/3}=2.51984\ldots,
\end{equation}
i.e.\ $3\in\operatorname{int}\mathcal D$; the zeros of
$x^4+y^4+z^4+1+3xyz$ on $\Tt^3$ form a nonempty real-analytic set (the
integrand of the direct Mahler integral is log-singular there; see
Section~\ref{sec:numerics}).

\begin{proposition}[No path continuation is possible]\label{prop:nopath}
Let $c(\tau)$ be the branch \eqref{eq:cdef}, with $c(\tau_2)=3$
\textup{(}from $s_4(\tau_2)=81$, \S\ref{sec:P1}\textup{)}. Then $\tau_2$
lies in the \emph{interior} of the bad preimage
$c^{-1}(\operatorname{int}\mathcal D)$; in particular $\tau_2$ is not in
the closure of the good region
$\mathcal G:=\{\tau\in\Hh:c(\tau)\notin\mathcal D\}$, and no continuous
path can approach $\tau_2$ through $\mathcal G$.
\end{proposition}
\begin{proof}
$c$ is holomorphic and non-constant near $\tau_2$. By the open mapping
theorem, $c$ maps every sufficiently small neighborhood of $\tau_2$
onto a neighborhood of $c(\tau_2)=3$, and since
$3\in\operatorname{int}\mathcal D$ that neighborhood lies inside
$\operatorname{int}\mathcal D$. Hence a full neighborhood of $\tau_2$
lies in the bad preimage.
\end{proof}

\begin{remark}[Contrast with the $n_2$-family]
In \cite{Zf} the certified-path criterion required exactly the
membership of the target CM point in the closure of the good component
$W$, supplied there by an explicit $99$-block interval certificate ---
possible because the bad slit $[0,64]$ has empty interior. Here the bad
locus has nonempty interior \emph{and} swallows the target: the
topological input of the continuation machine does not merely fail to be
certified, it is false. Figure~\ref{fig:astroid} shows the geometry.
\end{remark}

\begin{figure}[ht]
\centering
\begin{tikzpicture}
\begin{axis}[
  width=10.2cm, axis equal image, axis lines=middle, clip=false,
  xlabel={$\ReT c$}, ylabel={$\ImT c$},
  xmin=-5.4, xmax=6.6, ymin=-4.9, ymax=4.9,
  xtick={-4,4}, ytick={-4,4},
  ticklabel style={font=\small},
  every axis x label/.style={at={(ticklabel* cs:1.02)}, anchor=west},
  every axis y label/.style={at={(ticklabel* cs:1.02)}, anchor=south},
]
\addplot[blue, thick, fill=blue!7, domain=0:360, samples=241]
  ({4*(cos(x))^3},{4*(sin(x))^3}) \closedcycle;
\addplot[orange!90!black, dashed, very thick] coordinates {(-4.9,0)(5.6,0)};
\addplot[orange!90!black, dashed, very thick] coordinates {(0,-4.5)(0,4.5)};
\addplot[only marks, mark=*, mark size=1.7pt, red] coordinates {(3,0)};
\node[anchor=south west, red] at (axis cs:3.02,0.14) {$c=3$};
\addplot[only marks, mark=*, mark size=1.5pt] coordinates {(5.0397,0)};
\node[anchor=north] at (axis cs:5.05,-0.28) {\small $c(\ii)=648^{1/4}$};
\node[blue] at (axis cs:-3.35,3.55) {\small $3\e^{\ii t}+\e^{-3\ii t}$};
\node[orange!90!black] at (axis cs:2.6,-4.15) {\small cross
  $[-4,4]\cup\ii[-4,4]$ (the $\mt$ branch cut)};
\end{axis}
\end{tikzpicture}
\caption{The critical image $P(\Tt^3)$ of the $n_4$-family: the astroid
disc $\mathcal D$ of \eqref{eq:astroiddisc} (shaded), bounded by
$c=3\e^{\ii t}+\e^{-3\ii t}$. The parameter $c=3$ of the conjecture
$n_4(81)$ is an interior point, so no holomorphic path in $\tau$-space
can reach $\tau_2$ through the good region
(Proposition~\ref{prop:nopath}). For comparison, the proved anchor
$c(\ii)=648^{1/4}\approx5.04$ lies outside. The dashed cross is the
one-dimensional branch cut $c^4\in(0,256]$ of the Rogers ${}_5F_4$
branch of $\mt$ --- the bad set the $k=1$ machine of \cite{Zf} would
have had to avoid, had $\m=\ReT\mt$ held inside $\mathcal D$ off the
cross; Section~\ref{sec:sheet} shows it does not.}
\label{fig:astroid}
\end{figure}
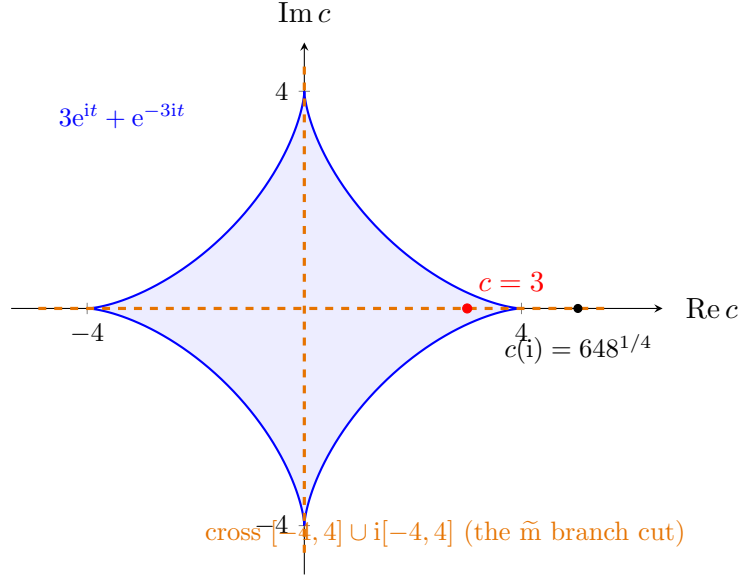

\begin{remark}[The refuted hope]
The branch cut of the holomorphic Mahler measure $\mt$ (via Rogers'
${}_5F_4$, with $c^4\in(0,256]$) is the \emph{cross}
$[-4,4]\cup\ii[-4,4]$, one-dimensional. Had the identity
$\m(P+c)=\ReT\mt(c)$ held inside $\mathcal D$ off the cross (the
Villegas criterion \cite{RV}), then $c=3$ would have been a boundary
point of the good set and the differential-comparison machine of
\cite{Zf} would have ported verbatim. The go/no-go experiment
\texttt{diag\_n4\_astroid.py} sampled $\tau$ with $c(\tau)$ inside the
astroid \emph{and off the cross} and compared $\mathrm{EK}_4(\tau)$
against direct integration: the identity fails there too
(differences $0.35$--$5.15$). The bad set is not the cross; see the next
section.
\end{remark}

\section{Obstruction 2: the \texorpdfstring{$U$}{U}-series leaves the
geometric sheet below \texorpdfstring{$\ImT\tau=1/\sqrt2$}{Im tau =
1/sqrt(2)}}\label{sec:sheet}

The decisive obstruction is analytic rather than topological. Samart's
identity \eqref{eq:samart4} is proved only for $\ImT\tau\ge1/\sqrt2$,
while
\[
  \ImT\tau_2=\frac{\sqrt7}4=0.66143\ldots<\frac1{\sqrt2}=0.70710\ldots .
\]
The $U$-series \eqref{eq:EK4U} nevertheless converges for every
$\tau\in\Hh$, defining $\mathrm{EK}_4$ as a single-valued real-analytic
function on all of $\Hh$; the function
$\tau\mapsto n_4(s_4(\tau))$, by contrast, is real-analytic only off the
preimage of the critical image. The question is where the two agree.

The diagnostic experiment of the project archive
(\texttt{diag\_n4\_astroid.py}, 2026-07-29; control checks first:
$\mathrm{EK}_4(\ii)=4\m(P+648^{1/4})=6.4332830658\ldots$ to
$5.4\times10^{-12}$, and conjugate symmetry exact) gives a clean
verdict:

\begin{itemize}
\item \emph{Below the threshold, the identity fails everywhere}, not
merely inside the astroid: at $\tau=0.3\ii$ one has
$c(\tau)=13.72\ldots$ (real, outside $\mathcal D$, where even Rogers'
series is valid) and $|\mathrm{EK}_4(\tau)-4\m(P+c(\tau))|=11.4\ldots$;
at interior non-cross sample points the differences range over
$0.35$--$5.15$, with $\mathrm{EK}_4$ taking \emph{negative} values
(while a Mahler measure is nonnegative).
\item \emph{The difference grows continuously from the boundary}: along
the vertical scan $\ReT\tau=1/4$, $\ImT\tau$ from $0.70$ down to $0.20$
($20$ points), $|\mathrm{EK}_4-4\m|$ grows continuously from $0$ at the
threshold --- no jump, hence no locally constant sheet offset that a
correction term could fix.
\end{itemize}

Figure~\ref{fig:departure} shows the same phenomenon on the imaginary
axis, computed afresh for this note (\texttt{gen\_fig\_n4.py},
\texttt{n4\_departure.dat}): the two curves agree for
$y\ge1/\sqrt2$ (as they must by Theorem~\ref{thm:samart4}) and separate
strictly below it, even though $c(\ii y)$ remains real and $\ge4$ --- in
particular outside the critical image --- on the whole axis.

\begin{figure}[ht]
\centering
\begin{tikzpicture}
\begin{axis}[
  width=12.6cm, height=8.6cm,
  xlabel={$y$}, xmin=0.28, xmax=1.52,
  ymin=3.3, ymax=23,
  legend pos=north east,
  legend style={font=\small},
]
\addplot[blue, thick] table[x=y, y=ek4] {n4_departure.dat};
\addlegendentry{$\mathrm{EK}_4(\ii y)$ \ ($U$-series \eqref{eq:EK4U})}
\addplot[red, thick, densely dashed] table[x=y, y=fourm]
  {n4_departure.dat};
\addlegendentry{$4\,\m\bigl(P+c(\ii y)\bigr)$ \ (direct torus
  integration)}
\draw[gray, dashed]
  ({axis cs:0.70710678,0}|-{rel axis cs:0,0}) --
  ({axis cs:0.70710678,0}|-{rel axis cs:0,1});
\node[gray, anchor=south west, rotate=90, font=\small] at
  ({axis cs:0.70710678,0}|-{rel axis cs:0,0.02})
  {$y=1/\sqrt2$ (Samart's threshold)};
\draw[violet, dotted, thick]
  ({rel axis cs:0,0}|-{axis cs:0,4.10686431}) --
  ({rel axis cs:1,0}|-{axis cs:0,4.10686431});
\node[violet, anchor=south east, font=\small] at
  ({rel axis cs:0.985,0}|-{axis cs:0,4.10686431})
  {$40M_7=4.10686\ldots=\mathrm{EK}_4(\tau_2)$
   (Theorem~\ref{thm:P1n4})};
\end{axis}
\end{tikzpicture}
\caption{The wrong-sheet departure on the imaginary axis
(data: \texttt{n4\_departure.dat}, generated by
\texttt{gen\_fig\_n4.py}; $61$ ordinates, $30$ dps, control
$\mathrm{EK}_4(\ii)=4\m=6.4332830658\ldots$). The $U$-series
$\mathrm{EK}_4(\ii y)$ coincides with the true Mahler value
$4\m(P+c(\ii y))$ for $y\ge1/\sqrt2$ and departs from it below the
threshold, although $c(\ii y)\ge4$ stays real and outside the critical
image on the whole range: the series has climbed to a companion sheet of
$\mt$. The horizontal dotted line marks
$40M_7=4.10686431115608\ldots$, the exact value of the series at the CM
point $\tau_2=(7+\sqrt{-7})/4$ (Theorem~\ref{thm:P1n4}), i.e.\ the
conjectured value of $n_4(81)$; it lies below every value the series
attains on the imaginary axis. The direct integration of
Section~\ref{sec:numerics} shows $n_4(81)\ne\mathrm{EK}_4(\tau_2)$:
the series value is on the wrong sheet.}
\label{fig:departure}
\end{figure}
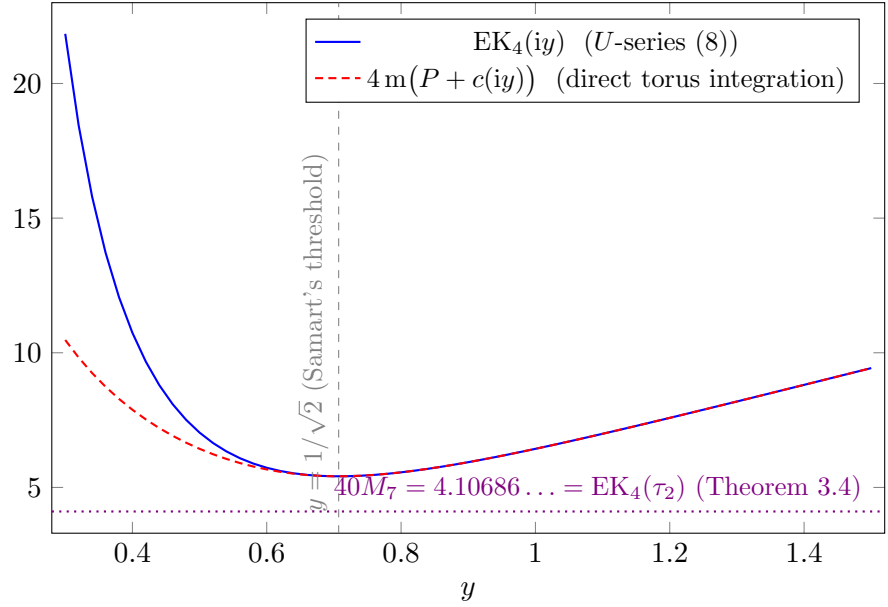

\begin{remark}[Mechanism, and why the $k=1$ machine does not port]
This is the same wrong-sheet phenomenon as at the point $\tau'$ of
Theorem~B in \cite{Zf} (Remark (wrong-sheet) there): $\mathrm{EK}$ is
single-valued on $\Hh$, and below Samart's threshold it continues on a
companion sheet of the multivalued $\mt$, with an exact closed-form
value on that sheet (for $n_2$: $(8/7)(44M_7-d_7)$ instead of
$n_2(s_2(\tau'))$). For Theorem~A the continuation machine still worked
because on the good component $W$ the series branch \emph{equals} the
geometric branch: the differential identity
$\Psi=\Omega(c)c'+E'\equiv0$ could spread from Samart's region by the
identity theorem, and the bad slit was avoidable. For the $n_4$-family
both premises fail simultaneously: the series branch differs from the
geometric branch everywhere below $1/\sqrt2$ (so $\Psi\equiv0$ is false
on any would-be $W$), and the critical image is a disc around the
target (Proposition~\ref{prop:nopath}). The Fricke shortcut of
Theorem~B is likewise unavailable: $\mathrm{EK}_4$ is $1$-periodic and
the evaluation at $\tau_2=\tau_w+2$ is on the wrong sheet already.
What remains is regulator/monodromy machinery in the style of He--Ye
\cite{HY} (who proved Samart's $f_3$ conjectures), or a direct regulator
evaluation at $c=3$ in the style of Zheng--Guo--Qin \cite{ZGQ}.
\end{remark}

\section{Numerical status of the conjecture}\label{sec:numerics}

\noindent A high-precision direct torus integration of
\[
  n_4(81)=4\,\m\bigl(x^4+y^4+z^4+1+3xyz\bigr)
\]
gives
\[
  n_4(81)=4.1655349907533676508(5),
\]
against the conjectured value
\[
  40M_7=4.1068643111560804484182637959317165599558\ldots
\]
(computed from $L(g_7,3)$ via the functional equation \eqref{eq:FE}).
The discrepancy is
\[
  n_4(81)-40M_7=+0.0586706795972872024\ldots,
\]
about $2\times10^{5}$ times the integrator's error floor
$2\times10^{-20}$ (outer tanh--sinh quadrature at degrees $5,6,7,8$,
successive variations $6\times10^{-21}$, $1.9\times10^{-20}$,
$4\times10^{-21}$; cross-validated against an independent numpy
trapezoid computation to $2\times10^{-8}$). Since $\ImT\tau_2
=\sqrt7/4<1/\sqrt2$, this is exactly the regime in which the $U$-series
is on the wrong sheet (Section~\ref{sec:sheet}), and the numerical
picture is consistent: the series value $\mathrm{EK}_4(\tau_2)=40M_7$
is exact (Theorem~\ref{thm:P1n4}) but is \emph{not} the true Mahler
measure.

\emph{As literally stated, the conjecture $n_4(81)\stackrel{?}{=}40M_7$
is therefore numerically refuted.} What remains open is a closed form
for the true value $n_4(81)$; the offset
$0.0586706795972872\ldots$ has no identified closed form. We stress
that this computation is high-precision numerical evidence, not an
interval-arithmetic certificate in the sense of \cite[\S6]{Zf}; with a
discrepancy five orders of magnitude above the error floor we regard
the conclusion as decisive.

\paragraph{The original evidence.} Samart's convention throughout
\cite{Sa15} is that ``$\stackrel{?}{=}$'' denotes equality to at least
$25$ decimal places, but no decimal value of $n_4(81)$ is printed
there. In view of Fei's caveat cited in \S1 and of the exact series
evaluation $\mathrm{EK}_4(\tau_2)=40M_7$ (Theorem~\ref{thm:P1n4}), the
original numerical evidence was presumably series-sided. The same
phenomenon in a two-variable family --- failure of such identities
inside the critical region ``because its Deninger path is not
closed'' --- has been studied explicitly by Samart and Tao \cite{ST};
the present example appears to be the first three-variable instance.

\paragraph{The integrator.} The integrand is log-singular ($c=3$ lies
inside the critical image, \S\ref{sec:astroid}). After the Jensen
reduction in $z$, the inner integrand $g(t_1,t_2)=\sum_j\log^+|z_j|$
has period $\pi/2$ in $t_1$ and symmetry about $\pi/4$, so
$\m=(4/\pi^2)\int_0^{\pi/4}(\int g\,dt_2)\,dt_1$. The kink locus
$\{|z_j|=1\}$ of the $z$-Jensen reduction is detected \emph{exactly}:
the resultant of $z^4+Az+B$ against its conjugate reciprocal is a
degree-$8$ polynomial in $W=y^4$, $X=x^4$ (generated once by
\texttt{kink\_resultant\_W.py}), whose unit-circle roots locate every
kink; only two critical $t_1$ occur in $(0,\pi/4)$, namely
$T_{1A}=0.361367123906707805589029\ldots$ (where two kink pairs merge
simultaneously) and $T_{1B}=\pi/6$. The inner integral is evaluated by
tanh--sinh quadrature at degree $7$ (degree $7$ vs $8$:
$1.3\times10^{-25}$). The computation is implemented in
\texttt{n4\_81\_final.py} ($42$--digit arithmetic).

\section{Conclusion}

As literally stated, the conjecture $n_4(81)=40M_7$ is numerically
refuted (Section~\ref{sec:numerics}): the two sides differ by
$+0.0586706795972872\ldots$, five orders of magnitude above the
integration error floor. This note has archived its status:

\begin{itemize}
\item the $L$-value half is proved exactly:
$\mathrm{EK}_4(\tau_2)=40M_7$ (Theorem~\ref{thm:P1n4}), with complete
algebra ($\cL_2=\OO_K$, $\cL_1=(\bar\varpi)/2$,
$-T_1+4T_2=28L(g_7,3)$, exact $\zeta_K(2)$-cancellation) and machine
certification ($45$ checks, worst $2.7\times10^{-59}$);
\item the Mahler half is blocked twice: the critical image is the
two-dimensional astroid disc with $c=3$ interior
(Propositions~\ref{prop:astroid} and~\ref{prop:nopath},
Figure~\ref{fig:astroid}), and the $U$-series is on the wrong sheet of
$\mt$ everywhere below $\ImT\tau=1/\sqrt2$
(Figure~\ref{fig:departure});
\item the numerical comparison of the two sides at $s=81$
(Section~\ref{sec:numerics}) exhibits the wrong-sheet offset directly:
$n_4(81)-40M_7=+0.0586706795972872\ldots$, with integration error
$\le 2\times10^{-20}$.
\end{itemize}

The $L$-value half of the conjecture is thus proved exactly, while the
literal identity fails: $\mathrm{EK}_4(\tau_2)=40M_7$ is the value of
the $U$-series on the wrong sheet of $\mt$, not the true Mahler
measure. The open problem that remains is a closed form for the true
value $n_4(81)$, for which the methods of the companion paper
\cite{Zf} are structurally insufficient and regulator/monodromy
techniques \cite{HY,ZGQ} appear necessary.

\paragraph{Script inventory.}
\texttt{verify\_P1\_n4\_81.py} (P1 certification, $45$ checks);
\texttt{diag\_n4\_astroid.py} (go/no-go diagnostic, wrong-sheet
quantitative data); \texttt{gen\_fig\_n4.py} +
\texttt{n4\_departure.dat} (Figure~\ref{fig:departure}, this paper);
\texttt{kink\_resultant\_W.py} + \texttt{n4\_81\_final.py} (exact kink
detection and direct integration at $c=3$,
Section~\ref{sec:numerics}).
The main-paper scripts certifying the shared ingredients
(\eqref{eq:hecke}, \eqref{eq:FE}) are listed in \cite[\S6]{Zf}.
All scripts of this note, together with their run logs, are available
in the public repository
\url{https://github.com/huiminZheng-collab/samart-mahler}, archived at
\url{https://doi.org/10.5281/zenodo.21711884}.

\end{document}